\documentclass{amsart}
\usepackage[normalem]{ulem}
\usepackage{amssymb,epsfig,mathrsfs,mathpazo,stackengine,scalerel,url,amsthm,thmtools,bbm}
\usepackage{xcolor}
\usepackage{stmaryrd}
\usepackage{epsfig,bm}
\usepackage[multiple]{footmisc}
\usepackage{accents} 
\usepackage{mathtools} 
\usepackage{footmisc}

\newcommand{\osc}{{\rm osc}}

\usepackage{arydshln, booktabs, makecell, pbox, tabularx}

\usepackage{graphicx, caption, subcaption}

\newtheorem{theorem}{Theorem}[section]

\newtheorem{proposition}[theorem]{Proposition}

\declaretheorem[style=remark,qed=$\vartriangle$,sibling=theorem]{remark}
\numberwithin{equation}{section}

\newcommand{\eps}{\varepsilon}

\newcommand{\R}{\mathbb R}

\newcommand{\N}{\mathbb N}

\newcommand{\cF}{\mathcal F}

\newcommand{\cE}{\mathcal E}

\newcommand{\cL}{\mathcal L}
\newcommand{\cS}{\mathcal S}
\newcommand{\Lis}{\cL\mathrm{is}}

\newcommand{\identity}{\mathrm{Id}}

\DeclareMathOperator{\ran}{ran}

\DeclareMathOperator*{\argmin}{argmin}

\DeclareMathOperator{\divv}{div}

\newcommand{\nrm}{| \! | \! |}

\newcommand{\new}[1]{{\color{black}{#1}}}

\newcommand{\be}{\begin{equation}}
\newcommand{\ee}{\end{equation}}

\newcommand{\tria}{{\mathcal T}}
\newcommand{\RT}{\mathit{RT}}

\newcommand{\uumlaut}{{\"u}}
\newcommand{\oumlaut}{{\"o}}

{\par\begin{samepage}%
\begin{tabbing}\ttfamily%
 \hspace*{5mm}\=\hspace{3ex}\=\hspace{3ex}\=\hspace{3ex}\=\hspace{3ex}%
\=\hspace{3ex}\=\hspace{3ex}\=\hspace{3ex}\=\hspace{3ex}\kill}%
{\end{tabbing}\end{samepage}}

\newcounter{ccondition}

{%
\renewcommand{\theequation}{\temp}%
\addtocounter{equation}{-1}%
\par\noindent%
\ignorespacesafterend%
}

\newcounter{mylistcounter}
\renewcommand{\themylistcounter}{(\roman{mylistcounter})}

\newenvironment{mylist}{
\begin{list}{\themylistcounter.}{\usecounter{mylistcounter}
\setlength{\labelwidth}{-\smallskipamount}
\setlength{\labelsep}{\medskipamount}
\setlength{\topsep}{\smallskipamount}
\setlength{\itemsep}{\smallskipamount}
\setlength{\itemindent}{0cm}
\setlength{\leftmargin}{0cm}}}
{\end{list}}

\makeatletter
\@namedef{subjclassname@2020}{%
  \textup{2020} Mathematics Subject Classification}
\makeatother

\title[Minimal residual methods with inhomogeneous boundary conditions]{A convenient inclusion of inhomogeneous boundary conditions in minimal residual methods}

\date{\today}

\author{Rob Stevenson}
\address{Korteweg-de Vries (KdV) Institute for Mathematics, University of Amsterdam, P.O. Box 94248, 1090 GE Amsterdam, The Netherlands.}
\email{rob.p.stevenson@gmail.com}

\thanks{This research has been supported by the NSF Grant DMS ID 1720297.}



\begin{document}

\begin{abstract} Inhomogeneous essential boundary conditions can be appended to a well-posed PDE to lead to a combined variational formulation.
The domain of the corresponding operator is a Sobolev space on the domain $\Omega$ on which the PDE is posed, whereas the codomain is a Cartesian product of spaces, among them fractional Sobolev spaces of functions on $\partial\Omega$. In this paper, easily implementable minimal residual discretizations are constructed which yield quasi-optimal approximation from the employed trial space, in which the evaluation of fractional Sobolev norms is fully avoided.
\end{abstract}
\maketitle

\section{Introduction} 
\subsection{MINRES methods for handling inhomogeneous (essential) boundary conditions}
The possibility to handle inhomogeneous boundary conditions when solving PDEs is often mentioned as an advantage of minimal residual (MINRES) discretisations (e.g.~\cite[Ch.~12]{23.5}).
In most cases, however, it is not so clear how this can lead to satisfactory results.

Considering linear elliptic PDEs of 2nd order, one option is to write the boundary value problem in an ultra-weak first order system variational formulation, which renders all boundary conditions natural.
The resulting `practical' MINRES method (\cite[\S4.4\&5]{204.19}) yields a quasi-best approximation from the trial space to all variables w.r.t.~$L_2$-norms.

If one is interested in approximation w.r.t.~other norms, then one can resort to first order system weak or mild-weak variational formulations, or to the standard second order variational formulation.
In those cases Dirichlet, or Dirichlet and Neumann boundary conditions are essential ones. 
In the case that homogeneous versions of those boundary conditions lead to a well-posed variational problem, inhomogeneous ones can be appended by enforcing them weakly by introducing additional test spaces of functions defined on the boundary $\partial\Omega$ of the domain $\Omega$ on which the PDE is posed. 
Then the combined variational formulation of the PDE and the boundary conditions can be shown to be well-posed.
Indeed, for some Hilbert spaces $X$, $Y_1$ and $Y_2$, let $F\colon X\rightarrow Y_1'$ and $T \colon X \rightarrow Y_2'$ be bounded, and with $X_0:=\ker T$, $F\colon X_0 \rightarrow Y_1'$ be boundedly invertible. 
In applications, $F$ and $T$ are operators corresponding to a weak imposition of the PDE and the boundary condition(s), and so $X$ and $Y_1$ are spaces of functions on $\Omega$, and $Y_2$ is a (product of) space(s) of functions on $\partial\Omega$.
Then assuming $T$ is surjective, $G:=(F,T)\colon X \rightarrow Y':=Y_1' \times Y_2'$ is boundedly invertible (\cite[Lemma 2.7]{75.28}).

Given a finite dimensional `trial' subspace $X^\delta \subset X$ (`$\delta$' stands for `discrete'), and a given forcing term and essential boundary data, 
the solution $u \in X$ of the well-posed problem can be approximated by 
the residual minimizer $u^\delta \in X^\delta$ in the norm of $Y'$.
A complication is that some coordinate spaces of $Y'$ will be negative- or, in particular for function spaces on $\partial\Omega$, fractional-order Sobolev spaces, all with norms that cannot be evaluated.
Several solutions for this problem have been proposed, e.g.~\cite{35.835} for the case of negative order Sobolev norms on $\Omega$, and \cite{249.06} for fractional order Sobolev norms on $\partial\Omega$.
The resulting MINRES methods, however, are not guaranteed to produce approximate solutions that are quasi-best.

Viewing all negative- or fractional-order Sobolev norms as dual norms, so writing e.g.~ $H^{\frac12}(\partial\Omega)$ as $H^{-\frac12}(\partial\Omega)'$, our approach in \cite{204.19} is to discretize dual norms by replacing the involved suprema over infinite dimensional spaces by a supremum over a finite dimensional (product) test space $Y^\delta \subset Y$.
Under an inf-sup condition on the pair $(X^\delta,Y^\delta)$, the resulting MINRES approximation  $u^\delta \in X^\delta$ was shown to be quasi-best.
The evaluation of the so-discretized dual norms is implemented by introducing the Riesz lift of the residual, viewed as an element of ${Y^\delta}'$, as an extra variable $\lambda^\delta$, which results in a saddle point system for $(\lambda^\delta,u^\delta) \in Y^\delta \times X^\delta$.

In the case that one or more components of $Y$ are fractional Sobolev norms, a remaining issue is the evaluation of the scalar product between functions from $Y^\delta$.
Without compromizing quasi-optimality, it can be solved by replacing this scalar product  on $Y^\delta$ by an equivalent scalar product defined in terms of the inverse of an optimal preconditioner of linear complexity for the corresponding stiffness matrix. This construction permits then to eliminate the extra variable from the system, after which a symmetric positive definite system on $X^\delta$ remains, with a system matrix that can be applied in linear complexity.
Optimal preconditioners of linear complexity for fractional Sobolev spaces are available, e.g. BPX or multigrid for positive orders, and \cite{75.258,249.985} for negative orders.

Finally, by applying an optimal preconditioner of linear complexity on the trial space $X^\delta$, the symmetric positive definite system can be solved iteratively within a tolerance of the order of the discretization error in linear complexity.

\subsection{Current paper} 
A disadvantage of the method from \cite{204.19} is that its implementation is rather demanding, in particular because of the incorporation of the preconditioner(s) for fractional Sobolev norms on the boundary. In the current paper, we introduce an alternative approach, which can be expected to be computationally more costly, but that can be very easily implemented with a finite element package like NGSolve \cite{247.065}.

As in our approach from \cite{204.19}, given a trial space $X^\delta \subset X$, we select a test space $Y^\delta \subset Y$ such that $(X^\delta,Y^\delta)$ is (uniformly) inf-sup stable; replace the norm on $Y'$ for the residual by the discretized norm on ${Y^\delta}'$; and introduce its Riesz lift in $Y^\delta$ as an extra variable resulting in a saddle-point problem.

We proceed differently to deal with the problem that also the scalar product on $Y$ might not be evaluable. Recalling that $G\colon X \rightarrow Y'$, and thus $G'\colon Y \rightarrow X'$ is boundedly invertible, we equip $Y$ with the equivalent norm $\|G'\cdot\|_{X'}$, known as the optimal test norm. The resulting dual norm on $Y'$ then equals $\|G^{-1}\cdot\|_X$, so that the (unfeasible) exact residual minimization in this norm would yield the best approximation from $X^\delta$ w.r.t.~$\|\cdot\|_X$. Since also the new scalar product $\langle G'\cdot,G'\cdot \rangle_{X'}$ on $Y$ cannot be evaluated, for some finite dimensional $\hat{X}^\delta \subset X$
we replace it by a discretized version $\langle G'\cdot,G'\cdot \rangle_{{\hat{X}^\delta}'}$.  Assuming that also $(Y^\delta,\hat{X}^\delta)$ satisfies a (uniform) inf-sup condition, the resulting MINRES approximation is quasi-best. 
To make this approximation computable, we introduce the Riesz lift $\theta^\delta \in \hat{X}^\delta$ of $G' \lambda^\delta$, viewed as an element of ${\hat{X}^\delta}'$, as a third variable, and so saddle the system once more.

The resulting $3 \times 3$ block system for the triple $(\theta^\delta,\lambda^\delta,u^\delta) \in \hat{X}^\delta \times Y^\delta \times X^\delta$ has one block that corresponds to the stiffness matrix of $\hat{X}^\delta$ w.r.t.~the``easy'' scalar product $\langle\cdot,\cdot\rangle_X$, whereas all four other non-zero blocks correspond to system matrices of the bilinear form $(G\cdot)(\cdot)$ w.r.t.~$X^\delta \times Y^\delta$ or $\hat{X}^\delta \times Y^\delta$, or transposes of those blocks.
So the ``difficult'' scalar product $\langle\cdot,\cdot\rangle_Y$ vanished completely from the system, and the implementation is simple.

For suitable $Y^\delta$ and $\hat{X}^\delta$, we show that $\|\theta^\delta\|_X$ is an efficient, and asymptotically reliable a posteriori estimator for $\|u-u^\delta\|_X$.
We will use local $X$-norms of $\theta^\delta$ to drive an adaptive finite element method.

\subsection{A nonintrusive approach}
Although is not subject of this paper, for completeness and comparison, in our abstract framework we recall an alternative classical approach to deal with inhomogeneous essential boundary data (e.g. \cite{75.08}).

Recall the setting of $T \colon X \rightarrow Y_2'$ being bounded and surjective, $F\colon X_0:=\ker T \rightarrow Y_1'$ being boundedly invertible, and 
$G:=(F,T)\colon X \rightarrow Y':=Y_1' \times Y_2'$ being boundedly invertible.
We will use that $T$ has a bounded right-inverse $E\colon Y_2' \rightarrow X$ (e.g. $E:= g \mapsto G^{-1}(0,g)$).

Given $(f,g) \in Y_1' \times Y_2'$, consider the problem $Gu=(f,g)$. 
Let $X^\delta$ be a finite dimensional subspace of $X$, and $X_0^\delta:=X^\delta \cap X_0$.
Suppose that one has a method available that for $g=0$ (i.e.~`homogeneous essential boundary data') produces a $u^\delta \in X^\delta$ that is quasi-best, i.e., for some constant $C>0$, $\|u-u^\delta\|_X \leq C \inf_{w^\delta \in X^\delta_0} \|u-w^\delta\|_X$. 
Using this method, one can approximate the solution $u$ for general $g$ as follows:
First, construct $g^\delta \in \ran T|_{X^\delta}$, i.e., $g^\delta = T v^\delta$ for some $v^\delta \in X^\delta$, that approximates $g$; and second, approximate with aforementioned method the solution $z \in X_0$ of $Gz=(f-F v^\delta,0)$ by $z^\delta \in X_0^\delta$. Then
\be \label{q1}
\begin{split}
\|u-(z^\delta&+v^\delta)\|_X \leq \|u-(z+v^\delta)\|_X+\|z-z^\delta\|_X\\
&\leq \|G^{-1}\|_{\cL(Y',X)}\|g-g^\delta\|_{Y_2'}+C\inf_{\{w^\delta\in X^\delta\colon Tw^\delta=g^\delta\}} \|z+v^\delta-w^\delta\|_X\\
&\leq C\inf_{\{w^\delta\in X^\delta\colon Tw^\delta=g^\delta\}} \|u-w^\delta\|_X +(1+C)  \|G^{-1}\|_{\cL(Y',X)}\|g-g^\delta\|_{Y_2'}.
\end{split}
\ee

Furthermore, following \cite[\S3.3]{15.4}, let us consider the case that there exists a uniformly bounded projector $J^\delta\colon X \rightarrow X^\delta$ that preserves `homogeneous essential boundary conditions', i.e., $T J^\delta (\identity -E T)=0$.
Then
$$
T(J^\delta u-J^\delta E(g-g^\delta))=T(J^\delta (u-ET u)+J^\delta ET v^\delta)=T v^\delta=g^\delta,
$$
and so
\be \label{q2}
\begin{split}
\inf_{\{w^\delta\in X^\delta\colon Tw^\delta=g^\delta\}} & \|u-w^\delta\|_X
\leq \|u-J^\delta u\|_X + \|J^\delta E(g-g^\delta)\|_X\\
&\leq \|J^\delta\|_{\cL(X,X)} \big(\inf_{w^\delta\in X^\delta} \|u-w^\delta\|_X+\|E\|_{\cL(Y_2',X)}\|g-g^\delta\|_{Y_2'}\big).
\end{split}
\ee
Finally, if $g^\delta:=P^\delta g$ for some uniformly bounded projector $P^\delta$, then
\be \label{q3}
\begin{split}
\|g-g^\delta\|_{Y_2'} &\leq \|P^\delta\|_{\cL(Y_2',Y_2')} \inf_{w^\delta\in X^\delta}\|Tu-Tw^\delta\|_{Y_2'} \\
& \leq \|P^\delta\|_{\cL(Y_2',Y_2')}\|T\|_{\cL(X,Y_2')}  \inf_{w^\delta\in X^\delta} \|u-w^\delta\|_X,
\end{split}
\ee
and by combining \eqref{q1},   \eqref{q2}, and  \eqref{q3}, we conclude that $z^\delta+v^\delta$ is a quasi-best approximation from $X^\delta$ to $u$.

An a posteriori estimator of the error in $u-(z^\delta+v^\delta)$ must control the error in both $z-z^\delta$ in $X$ and in $g-g^\delta$ in $Y_2'$.
Reliable error estimators for the latter term have been developed in \cite{168.861} (for Dirichlet data) and \cite{35.8553} (for Neumann data), but require additional smoothness of $g$ beyond being in $Y_2'$ (Dirichlet and Neumann data are required to be in $H^1$- or $L_2$-spaces on the boundary).

\subsection{Outline}
In Section~\ref{sec:2} we recall the principle of a MINRES discretization, and recall approaches to deal with the situation when
 the residual is measured in a dual norm $\|\cdot\|_{Y'}$, and when possibly also the norm on $Y$ cannot be evaluated.

In Section~\ref{sec:3}, we present an alternative solution for the second problem by equipping $Y$ with the optimal test norm. In addition, we discuss a posteriori error estimation.

The method from Section~\ref{sec:3} requires two uniform inf-sup conditions to be valid. In Section~\ref{sec:4} we discuss them for four different well-posed variational formulations of Poisson's problem with (generally inhomogeneous) Dirichlet and/or Neumann boundary conditions.

Numerical results are presented in Section~\ref{sec:5}.

\section{Least squares or minimal residual discretizations} \label{sec:2}
\subsection{Well-posed operator equation} 
For some Hilbert spaces $X$ and $Y$, for convenience over $\R$, an operator $G \in \Lis(X,Y')$, and an $f \in Y'$, we consider the equation
\be \label{system}
G u = f.
\ee
With the notation $G \in \Lis(X,Y')$, we mean that $G$ is a boundedly invertible linear operator $X \rightarrow Y'$, i.e., $G \in \cL(X,Y')$ and $G^{-1} \in \cL(Y',X)$.
\new{In the following, $u$ will always denote the solution of \eqref{system}.}

For any closed, in applications finite dimensional subspace $\{0\} \subsetneq X^\delta \subsetneq X$
from a family $\{X^\delta\}_\delta$ of such subspaces\footnote[4]{We impose the harmless conditions $X^\delta \neq \{0\}$, $X^\delta \neq X$ because at several occasions we will use that in a Hilbert space the norm of a projector $P \not\in\{ \identity,0\}$ is equal to the norm of $\identity-P$ (\cite{169.5,315.7}).}, consider the minimal residual or least squares approximation
\be \label{eq:exact}
u^\delta:=\argmin_{w \in X^\delta} \tfrac12 \|G w -f\|^2_{Y'}.
\ee

\subsection{Discretizing \new{the norm on $Y'$}, and saddle-point formulation} \label{s1}
Unless $Y$ is such that the Riesz map $R_Y\colon Y' \rightarrow Y$ can be efficiently evaluated, i.e., $Y$ is a (Cartesian product of) $L_2$-space(s), 
the minimizer $u^\delta$ from \eqref{eq:exact} is not computable.

Therefore, for a closed, in applications finite dimensional subspace $\{0\} \subsetneq Y^\delta=Y^\delta(X^\delta) \subsetneq Y$ with
\be \label{gammadelta}
\gamma^\delta := \inf_{0 \neq w \in X^\delta}\frac{\sup_{0 \neq v \in Y^\delta} \frac{|(G w)(v)|}{\|v\|_{Y}}}{\|G w\|_{Y'}}>0,
\ee
we \emph{replace} above $u^\delta$ by
\be \label{eq:minresprac2}
u^\delta:=\argmin_{w \in X^\delta} \tfrac12 \sup_{0 \neq v \in Y^\delta} \frac{|(G w-f)(v)|^2}{\|v\|_{Y}^2}.
\ee
Assuming $\inf_\delta \gamma^\delta >0$,  the following theorem shows that $u^\delta$ is a \emph{quasi-best} approximation to $u$ from $X^\delta$. 
\begin{theorem}[{\cite[Thm.~3.3]{204.19}}]  \label{thm:quasi-opt} Setting $\nrm \cdot\nrm_X:=\|G \cdot\|_{Y'}$ on $X$, for  $u^\delta$ from \eqref{eq:minresprac2} it holds that
$$
 \inf_{u \in X \setminus X^\delta} \frac{\inf_{w \in X^\delta}\nrm u-w\nrm_X}{\nrm u-u^\delta\nrm_X} \,\,=\,\, \gamma^\delta.
$$
\end{theorem}

The solution $u \in X$ of $G u = f$ is equal to $\argmin_{w \in X} \tfrac12 \|G w -f\|^2_{Y'}$, and so it solves the corresponding
 Euler-Lagrange equations
$$
\langle f-G u, G \undertilde{u}\rangle_{Y'}=0 \quad (\undertilde{u} \in X).
$$
Introducing $\lambda:=R_Y(f-Gu)$, the pair $(\lambda,u) \in Y \times X$ is the unique solution of
$$
 \langle \lambda, \undertilde{\lambda}\rangle_Y + (G u)(\undertilde{\lambda})  + (G \undertilde{u})(\lambda) = f(\undertilde{\lambda}) \quad ((\undertilde{\lambda},\undertilde{u}) \in Y \times X).
 $$
 
 Completely analogously, the minimal residual approximation $u^\delta$ from \eqref{eq:minresprac2} can be computed as the second component of the pair $(\lambda^\delta,u^\delta) \in Y^\delta \times X^\delta$ that solves
 \be \label{eq:minresprac3}
 \langle \lambda^\delta, \undertilde{\lambda}^\delta\rangle_Y + (G u^\delta)(\undertilde{\lambda}^\delta)  + (G \undertilde{u}^\delta)(\lambda^\delta) = f(\undertilde{\lambda}^\delta) \quad ((\undertilde{\lambda}^\delta,\undertilde{u}^\delta) \in Y^\delta \times X^\delta).
 \ee

\subsection{Changing the norm on $Y^\delta$} \label{s2} In several applications, not only $\|\cdot\|_{Y'}$ but also $\langle \cdot,\cdot\rangle_Y$ cannot be efficiently evaluated.
This occurs for example when $Y$ is a Cartesian product of spaces with one or more of them being fractional Sobolev spaces.
Therefore, let $\langle \cdot,\cdot\rangle_{Y^\delta}$ be a (efficiently evaluable) scalar product on $Y^\delta$, whose corresponding norm $\|\cdot\|_{Y^\delta}$ satisfies, for some $0<m^\delta\leq M^\delta<\infty$,
\be \label{normequiv}
m^\delta \|\cdot\|_{Y^\delta}^2 \leq \|\cdot\|_Y^2 \leq M^\delta\|\cdot\|_{Y^\delta}^2  \quad \text{on } Y^\delta.
\ee
Now we \emph{replace} $u^\delta$ from \eqref{eq:minresprac2} by
$$
u^\delta:=\argmin_{w \in X^\delta} \tfrac12 \sup_{0 \neq v \in Y^\delta} \frac{|(G w-f)(v)|^2}{\|v\|_{Y^\delta}^2},
$$
being the second component of the pair $(\lambda^\delta,u^\delta) \in Y^\delta \times X^\delta$ that solves
\be \label{eq:minresprac4}
 \langle \lambda^\delta, \undertilde{\lambda}^\delta\rangle_{Y^\delta} + (G u^\delta)(\undertilde{\lambda}^\delta)  + (G \undertilde{u}^\delta)(\lambda^\delta) = f(\undertilde{\lambda}^\delta) \quad ((\undertilde{\lambda}^\delta,\undertilde{u}^\delta) \in Y^\delta \times X^\delta).
 \ee

Assuming $\inf_\delta \gamma^\delta >0$, and $\sup_\delta M^\delta/m^\delta <\infty$, the next result shows that this $u^\delta$ is a \emph{quasi-best} approximation to $u$ from $X^\delta$. 

\begin{proposition}[{\cite[Thm.~3.6]{204.19}}] \label{prop1} For the solution $(\lambda^\delta,u^\delta) \in Y^\delta \times X^\delta$ of \eqref{eq:minresprac4}, it holds that
$$
\nrm u-u^\delta\nrm_X \leq \frac{\sqrt{M^\delta}}{\gamma^\delta \sqrt{m^\delta}} \inf_{w \in X^\delta}\nrm u-w\nrm_X.
$$
\end{proposition}

Concerning the selection of $\|\cdot\|_{Y^\delta}$, let $K^\delta={K^\delta}' \in \Lis({Y^\delta}',Y^\delta)$ be an operator whose application can be computed efficiently.
Such an operator is called a \emph{preconditioner} for $A^\delta \in \Lis(Y^\delta,{Y^\delta}')$ defined by $(A^\delta v)(\undertilde{v}):=\langle v,\undertilde{v}\rangle_Y$.
Setting
$$
\langle v,\undertilde{v}\rangle_{Y^\delta}:=((K^\delta)^{-1} v)(\undertilde{v}) \quad (v,\undertilde{v} \in Y^\delta),
$$
the corresponding norm $\|\cdot\|_{Y^\delta}$ satisfies \eqref{normequiv} with $m^\delta=\lambda_{\min}(K^\delta A^\delta)$ and $M^\delta=\lambda_{\max}(K^\delta A^\delta)$.
By substituting above choice for $\langle \cdot,\cdot\rangle_{Y^\delta}$ in \eqref{eq:minresprac4}, and by subsequently eliminating $\lambda^\delta$ from the saddle-point system, one infers that $u^\delta \in X^\delta$ can be computed as the solution of the \emph{symmetric positive definite} system
$$
(G \tilde{u}^\delta)(K^\delta(G u^\delta-f))=0 \quad (\tilde{u}^\delta \in X^\delta).
$$

\section{Equipping $Y$ with the optimal test norm}  \label{sec:3}
 \subsection{Optimal test norm} \label{s3} 
 In applications the solution proposed in Sect.~\ref{s2} to circumvent the evaluation \new{of} a `difficult' scalar product $\langle\cdot,\cdot\rangle_Y$ by means of the introduction 
 of a preconditioner may require quite some efforts concerning coding. This holds true when $Y$ is a Cartesian product of spaces with one or more of them being fractional Sobolev spaces on a manifold.
 In the current section we give an alternative for this approach by replacing the canonical norm on $Y$ by the so-called optimal test norm.
 
As a consequence of $G \in \Lis(X,Y')$, it holds that $G' \in \Lis(Y,X')$.
 From here on we \emph{replace} the canonical norm on $Y$ by the equivalent norm
 $$
 \|G'\cdot\|_{X'},
 $$
 and correspondingly, equip $Y'$ with the resulting dual norm
 $$\sup_{0 \neq v \in Y}\frac{|\cdot(v)|}{\|G'v\|_{X'}}=\sup_{0 \neq g \in X'}\frac{|g(G^{-1}\cdot)|}{\|g\|_{X'}}=\|G^{-1}\cdot\|_X.
 $$
 We conclude that w.r.t.~these new norms on $Y$ and $Y'$, $G':Y \rightarrow X'$ and $G:X \rightarrow Y'$ are \emph{isometries}. For this reason $\|G'\cdot\|_{X'}$ is known as the \emph{optimal test norm} on $Y$ (\cite{19.83,320.7, 58.3,35.8565}).
 
 \subsection{Discretizing the norm on $Y'$, and saddle-point formulation}
 Also with the new norm on $Y'$, the minimizer in \eqref{eq:exact} cannot be computed.
 Following Sect.~\ref{s1}, writing this norm in dual form $\sup_{0 \neq v \in Y}\frac{|\cdot(v)|}{\|G'v\|_{X'}}$, we discretize it 
 by supremizing over $Y^\delta\setminus\{0\}$, and rewrite the resulting least-squares problem in saddle-point form.
 
 So we consider
$$
u^\delta:=\argmin_{w \in X^\delta} \tfrac12 \sup_{0 \neq v \in Y^\delta} \frac{|(G w-f)(v)|^2}{\|G' v\|_{X'}^2},
$$
being the second component of $(\lambda^\delta,u^\delta) \in Y^\delta \times X^\delta$ that solves
\be \label{eq:minresprac5}
\langle G'\lambda^\delta, G'\undertilde{\lambda}^\delta\rangle_{X'} + (G u^\delta)(\undertilde{\lambda}^\delta)  + (G \undertilde{u}^\delta)(\lambda^\delta) = f(\undertilde{\lambda}^\delta) \quad ((\undertilde{\lambda}^\delta,\undertilde{u}^\delta) \in Y^\delta \times X^\delta).
\ee
Theorem~\ref{thm:quasi-opt} applies, where now $\nrm\cdot\nrm_X$ reads as the canonical norm $\|\cdot\|_X$, i.e.,
$$
\inf_{u \in X \setminus X^\delta} \frac{\inf_{w \in X^\delta}\| u-w\|_X}{\| u-u^\delta\|_X} \,\,=\,\, \gamma^\delta,
$$
with the definition of $\gamma^\delta$ reading as
 \be \label{gammadeltanew}
\gamma^\delta := \inf_{0 \neq w \in X^\delta}\frac{\sup_{0 \neq y \in Y^\delta} \frac{|(G w)(y)|}{\|G'y\|_{X'}}}{\|w\|_{X}}>0.
\ee

 \subsection{Discretizing \new{the norm on $X'$}, and saddling the system once more}
Unless $X$ is such that the Riesz map $R_X\colon X' \rightarrow X$ can be efficiently evaluated, i.e., $X$ is a (Cartesian product of) $L_2$-space(s),  as in Sect.~\ref{s2} we are  in the situation that the norm on $Y^\delta$, here $\|G' \cdot\|_{X'}$, cannot be evaluated, so that
\eqref{eq:minresprac5} does not correspond to an implementable method.
Similar to Sect.~\ref{s2}, we will therefore replace this norm by an equivalent one.

Let $\{0\} \subsetneq \hat{X}^\delta =\hat{X}^\delta(Y^\delta) \subsetneq X$ be a finite dimensional subspace for which
\be \label{mudelta}
\mu^\delta:=\inf_{0 \neq v \in Y^\delta}
\frac{\sup_{0 \neq w \in \hat{X}^\delta} \frac{|(G' v)(w)|}{\|w\|_{X}}}
{\|G' v\|_{X'}}>0.\,\,\,\footnotemark[2]
\ee%
\footnotetext[2]{Notice that \eqref{mudelta} is of the form \eqref{gammadelta} by reading $(G,X,Y,X^\delta,Y^\delta)$ as $(G',Y,X,Y^\delta,\hat{X}^\delta)$.}%
Setting on ${\hat{X}^\delta}' \supset X'$, $\|\cdot\|_{{\hat{X}^\delta}'}:=\sup_{0 \neq w \in \hat{X}^\delta} \frac{|\cdot(w)|}{\|w\|_{X}}$, and denoting with $\langle \cdot,\cdot\rangle_{{\hat{X}^\delta}'}$ the corresponding bilinear form, \new{we have 
\be \label{normequiv2}
\|G'\cdot\|^2_{{\hat{X}^\delta}'} \leq \|G'\cdot\|^2_{X'} \leq \frac{1}{(\mu^\delta)^2} \|G'\cdot\|^2_{{\hat{X}^\delta}'} \quad \text{on } Y^\delta,
\ee
being the counterpart of \eqref{normequiv}. As in Sect.~\ref{s2}, we now \emph{replace} \eqref{eq:minresprac5} by the problem of finding $(\lambda^\delta,u^\delta) \in Y^\delta \times X^\delta$ that solves
\be \label{eq:minresprac6}
\langle G'\lambda^\delta, G'\undertilde{\lambda}^\delta\rangle_{{\hat{X}^\delta}'} + (G u^\delta)(\undertilde{\lambda}^\delta)  + (G \undertilde{u}^\delta)(\lambda^\delta) = f(\undertilde{\lambda}^\delta) \quad ((\undertilde{\lambda}^\delta,\undertilde{u}^\delta) \in Y^\delta \times X^\delta).
\ee

Recalling that with our current norm $\|G'\cdot\|_{X'}$ on $Y$, the norm $\nrm\cdot\nrm_X$ reads as $\|\cdot\|_X$, an application of Proposition~\ref{prop1} shows the following.}

\begin{proposition} \label{prop2} With $\gamma^\delta$ and $\mu^\delta$ as defined in \eqref{gammadeltanew} and \eqref{mudelta}, for $u^\delta \in X^\delta$ as defined in \eqref{eq:minresprac6} it holds that
$$
\|u-u^\delta\|_X \leq \frac{1}{\gamma^\delta \mu^\delta} \inf_{w \in X^\delta} \|u-w\|_X.
$$
\end{proposition}

To turn \eqref{eq:minresprac6} into an equivalent evaluable system, we introduce an additional variable.
With the Riesz map $R_{\hat{X}^\delta}\colon {\hat{X}^\delta}' \rightarrow \hat{X}^\delta$, let $\theta^\delta:=R_{\hat{X}^\delta} G' \lambda^\delta$, i.e.,
$$
(G' \lambda^\delta)(\undertilde{\theta}^\delta)=\langle \theta^\delta,\undertilde{\theta}^\delta\rangle_X \quad(\undertilde{\theta}^\delta \in \hat{X}^\delta).
$$
Then 
$$
\langle G'\lambda^\delta, G'\undertilde{\lambda}^\delta\rangle_{{\hat{X}^\delta}'}=\langle R_{\hat{X}^\delta}G'\lambda^\delta,R_{\hat{X}^\delta}G'\undertilde{\lambda}^\delta\rangle_X=(G' \undertilde{\lambda}^\delta)(\theta^\delta) \quad (\undertilde{\lambda}^\delta \in Y^\delta).
$$
By substituting the latter relation  into \eqref{eq:minresprac6}, and by adding the preceding equation which defined $\theta^\delta$, we arrive at the equivalent problem of finding $(\theta^\delta,\lambda^\delta,u^\delta) \in \hat{X}^\delta \times Y^\delta\times X^\delta$ that satisfies
$$
\langle \theta^\delta,\undertilde{\theta}^\delta\rangle_X -(G\undertilde{\theta}^\delta)(\lambda^\delta)-
(G\theta^\delta)( \undertilde{\lambda}^\delta)- (G u^\delta)(\undertilde{\lambda}^\delta)  - (G \undertilde{u}^\delta)(\lambda^\delta) = -f(\undertilde{\lambda}^\delta) 
$$
for all $(\undertilde{\theta}^\delta,\undertilde{\lambda}^\delta,\undertilde{u}^\delta) \in \hat{X}^\delta \times Y^\delta \times X^\delta$, 
or, in equivalent block-form,
\be \label{eq:minresprac8}
\left\{
\begin{alignedat}{5}
&\langle \theta^\delta,\undertilde{\theta}^\delta\rangle_X & & -(G\undertilde{\theta}^\delta)(\lambda^\delta) && &&=0 && (\undertilde{\theta}^\delta \in \hat{X}^\delta),\\
-&(G\theta^\delta)( \undertilde{\lambda}^\delta) && &&- (G u^\delta)(\undertilde{\lambda}^\delta) &&=- f(\undertilde{\lambda}^\delta) &\quad& (\undertilde{\lambda}^\delta \in Y^\delta),\\
& && -(G \undertilde{u}^\delta)(\lambda^\delta) && &&=0 && (\undertilde{u}^\delta  \in X^\delta).
\end{alignedat}
\right.
\ee

Notice that in comparison to \eqref{eq:minresprac3} the `difficult' scalar product $\langle\cdot,\cdot\rangle_Y$ completely disappeared from the system, and that apart from the usually `easy' scalar product $\langle\cdot,\cdot\rangle_X$, only the bilinear form $(G\cdot)(\cdot)$ has to be implemented.
To satisfy the inf-sup conditions $\gamma^\delta>0$ and $\mu^\delta>0$, the auxiliary spaces $Y^\delta$ and $\hat{X}^\delta$ have to be sufficiently large (in any case $\dim \hat{X}^\delta \geq \dim Y^\delta \geq \dim X^\delta$ is needed), which makes solving \eqref{eq:minresprac8} computationally relatively expensive. On the other hand, the implementation of the method is quite simple, whereas Proposition~\ref{prop2} shows that for `large' $Y^\delta$ and $\hat{X}^\delta$, the obtained solution will be close to the best approximation from $X^\delta$. 
Indeed, notice that for $Y^\delta=Y$ and $\hat{X}^\delta=X$, $u^\delta$ is the best approximation to $u$ from $X^\delta$ w.r.t.~$\|\cdot\|_X$, whereas the second line in 
 \eqref{eq:minresprac8} shows that then  $\theta^\delta=u-u^\delta$.

\subsection{A posteriori error estimation} \label{sec:apost}
As is well-known, $\gamma^\delta>0$ is guaranteed by existence of a $\Pi^\delta\in \cL(Y,Y^\delta)$ with
\be \label{fortin} (G X^\delta)(\ran (\identity -\Pi^\delta))=0,
\ee
where $\gamma^\delta  \geq \|G' \Pi^\delta {G'}^{-1}\|_{\cL(X',X')}^{-1}$; and conversely, $\gamma^\delta>0$ guarantees existence of such a `Fortin interpolator', being even a projector onto $Y^\delta$, with 
$\|G' \Pi^\delta {G'}^{-1}\|_{\cL(X',X')}=1/\gamma^\delta$ (e.g.~\cite[Lemma 26.9]{70.98} or \cite[Prop.~5.1]{249.992}).

In the following let $\Pi^\delta$ be a Fortin interpolator with $\|G' \Pi^\delta {G'}^{-1}\|_{\cL(X',X')} \eqsim 1/\gamma^\delta$ (so that $\inf_{\delta} \gamma^\delta>0$ is equivalent to $\sup_\delta \|G' \Pi^\delta {G'}^{-1}\|_{\cL(X',X')}<\infty$).

\begin{proposition} With $(\theta^\delta,\lambda^\delta,u^\delta) \in \hat{X}^\delta \times Y^\delta\times X^\delta$ being the solution of 
\eqref{eq:minresprac8}, and $\osc^\delta(f):=\|G^{-1}(\identity-\Pi^\delta)'f\|_{X}$, it holds that
\be \label{apost}
\mu^\delta \|\theta^\delta\|_X \leq \|u-u^\delta\|_X \leq \|G' \Pi^\delta {G'}^{-1}\|_{\cL(X',X')}  \|\theta^\delta\|_X +\osc^\delta(f).
\ee
\end{proposition}

\begin{proof}
Thanks to $\theta^\delta=R_{\hat{X}^\delta} G' \lambda^\delta$, we have
\be \label{theta}
\|\theta^\delta\|_X=\sup_{0 \neq v^\delta \in Y^\delta}\frac{|(G' v^\delta)(\theta^\delta)|}{\|G' v^\delta\|_{{\hat{X}^\delta}'}}=
\sup_{0 \neq v^\delta \in Y^\delta}\frac{|(f-Gu^\delta)(v^\delta)|}{\|G' v^\delta\|_{{\hat{X}^\delta}'}}.
\ee

Using \eqref{fortin} and the left inequality in \eqref{normequiv2}, we find that for $v \in Y$,
\begin{align*}
|(f&-G u^\delta)(v)| \leq |(f-G u^\delta)(\Pi^\delta v)| +|f((\identity-\Pi^\delta)v)| \\
& \leq \|G' \Pi^\delta v\|_{X'} \sup_{0 \neq v^\delta \in Y^\delta}\frac{|(f-Gu^\delta)(v^\delta)|}{\|G' v^\delta\|_{X'}} +|(G'v)(G^{-1}(\identity -\Pi^\delta)' f)|\\
& \leq \|G' \Pi^\delta v\|_{X'} \sup_{0 \neq v^\delta \in Y^\delta}\frac{|(f-Gu^\delta)(v^\delta)|}{\|G' v^\delta\|_{{\hat{X}^\delta}'}} +|(G'v)(G^{-1}(\identity -\Pi^\delta)' f)|\\
& \leq \Big(\|G' \Pi^\delta {G'}^{-1}\|_{\cL(X',X')} \|\theta^\delta\|_X+\|G^{-1}(\identity -\Pi^\delta)' f\|_{X}
\Big) \|G' v\|_{X'}.
\end{align*}
From $\sup_{0 \neq v \in Y}\frac{|(f-Gu^\delta)(v)|}{\|G' v\|_{X'}}=\|u-u^\delta\|_X$, we conclude the right inequality in \eqref{apost}.
 
 The left inequality in \eqref{apost} follows from the right inequality in  \eqref{normequiv2} and again \eqref{theta}.
\end{proof}

\begin{remark}[Bounding the oscillation term]
By taking $\Pi^\delta$ to be the Fortin projector with $\|G' \Pi^\delta {G'}^{-1}\|_{\cL(X',X')}=1/\gamma^\delta$, we find that
\begin{align*}
&\osc^\delta(f)=\sup_{0 \neq v \in Y} \frac{|((\identity-\Pi^\delta)'f)(v)|}{\|G' v\|_{X'}}=
\sup_{0 \neq v \in Y} \frac{|f((\identity-\Pi^\delta)v)|}{\|G' v\|_{X'}}\\
&\new{=}\sup_{0 \neq v \in Y} \inf_{w^\delta \in X^\delta} \frac{|G(u-w^\delta)((\identity-\Pi^\delta)v)|}{\|G' v\|_{X'}}=
\sup_{0 \neq v \in Y} \inf_{w^\delta \in X^\delta} \frac{|(G'(\identity-\Pi^\delta)v)(u-w^\delta)|}{\|G' v\|_{X'}}
\end{align*}
and so $\osc^\delta(f) \leq \frac{1}{\gamma^\delta} \inf_{w^\delta \in X^\delta}\|u-w^\delta\|_X$.

Even better is when $Y^\delta$ is selected such that it allows for the construction of a (uniformly bounded) Fortin interpolator such that
$\osc^\delta(f)$ is of higher order than $\inf_{w^\delta \in X^\delta}\|u-w^\delta\|_X$.
Then, in any case asymptotically, besides being efficient the estimator $\|\theta^\delta\|_X$  is also reliable.
\end{remark}

The derivation of the a posteriori error estimator in this subsection is similar to that in \cite{35.93556} specialized to the use of the optimal test norm on $Y$.
Modifications were needed because of the replacement on $Y^\delta$ of the optimal test norm $\|G'\cdot\|_{X'}$ by $\|G'\cdot\|_{{\hat{X}^\delta}'}$, and the introduction of the extra variable $\theta^\delta$.
In \cite{35.93556} the $Y$-norm of the approximate Riesz-lift $\lambda^\delta \in Y^\delta$ of the residual $f - G u^\delta$ was used as the a posteriori error estimator, whereas we use the $X$-norm of the approximate error $\theta^\delta \in \hat{X}^\delta$.
Local $X$-norms of $\theta^\delta \in \hat{X}^\delta$ will turn out to be effective for driving an adaptive finite element method.

\section{Applications}  \label{sec:4}
\subsection{Model elliptic second order boundary value problem} \label{sec:model}
On a bounded Lipschitz domain $\Omega \subset \R^d$, where $d \geq 2$, and closed $\Gamma_D, \Gamma_N \subset \partial\Omega$, with $\Gamma_D \cup \Gamma_N =\partial\Omega$ and $|\Gamma_D \cap \Gamma_N|=0$, consider the following elliptic second order boundary value problem
\begin{equation} \label{bvp}
 \left\{
\begin{array}{r@{}c@{}ll}
-{\rm div}\, A \nabla u+ B u &\,\,=\,\,& g &\text{ on } \Omega,\\
u &\,\,=\,\,& h_D &\text{ on } \Gamma_D,\\
\vec{n}\cdot A \nabla u &\,\,=\,\,& h_N &\text{ on } \Gamma_N,
\end{array}
\right.
\end{equation}
 where $B \in \cL(H^1(\Omega),L_2(\Omega))$, and $A(\cdot)=A(\cdot)^\top \in L_\infty(\Omega)^{d\times d}$ satisfies $\xi^\top A(\cdot) \xi \eqsim \|\xi\|^2$ ($\xi \in \R^d$).
We assume that the matrix $A$, and the first order operator $B$ are such that
$$
w\mapsto (v \mapsto \int_\Omega A \nabla w \cdot \nabla v +B w \, v \,dx)
\in \Lis\big(H^1_{0,\Gamma_D}(\Omega),H^1_{0,\Gamma_D}(\Omega)'\big).\,\,\,\footnotemark[3]
$$
\footnotetext[3]{When $\Gamma_D=\emptyset$, it can be needed to replace $H^1_{0,\Gamma_D}(\Omega)=H^1(\Omega)$ by $H^1(\Omega)/\R$. For simplicity, we do not consider this situation.}%

\subsection{Well-posed variational formulations} \label{sec:wellposed}
In \cite{249.96, 204.19}, the following variational formulations \ref{first}--\ref{fourth} of \eqref{bvp} have been shown to be well-posed. From the formulations given, one easily derives the expressions for `$G$', `$u$', `$f$', `$X$', and `$Y$'. Implicitly we will assume that the data $g$, $h_D$, and $h_N$ are such that `$f$' is in dual of `$Y$'.
In the case of homogeneous essential boundary conditions, the formulations \ref{first}--\ref{third} can be simplified by incorporating such conditions in the definition of the domain $X$, which option we do not consider here.

\begin{mylist}
\item({\bf 2nd order weak formulation}) \label{first} Find $u \in H^1(\Omega)$ such that
$$
\int_\Omega A \nabla u \cdot \nabla v_1 +B u \, v_1 \,dx+\int_{\Gamma_D} u v_2\,ds=g(v_1)+\int_{\Gamma_N} h_N v_1\,ds+ \int_{\Gamma_D} h_D v_2\,ds
$$
for all $(v_1,v_2)\in H^1_{0,\Gamma_D}(\Omega)\times \widetilde{H}^{-\frac12}(\Gamma_D)$.
\item({\bf 1st order mild formulation}) \label{second}
Find $(\vec{p},u) \in H(\divv;\Omega) \times H^1(\Omega)$ such that
\begin{align*}
\int_\Omega (\vec{p}-A \nabla u)\cdot \vec{v}_1+(B u-\divv \vec{p}) v_2\,dx&+\int_{\Gamma_D} u v_3\,ds+\int_{\Gamma_N} \vec{p}\cdot\vec{n} \,v_4\,ds\\
&=\int_\Omega g v_2\,dx+\int_{\Gamma_D} h_D v_3\,ds+\int_{\Gamma_N} h_N v_4 \,ds
\end{align*}
for all $(\vec{v}_1,v_2,v_3,v_4) \in L_2(\Omega)^d \times L_2(\Omega)\times \widetilde{H}^{-\frac12}(\Gamma_D)\times H_{00}^{\frac12}(\Gamma_N)$.
\item({\bf 1st order mild-weak formulation}) \label{third}
Find $(\vec{p},u) \in L_2(\Omega)^d \times H^1(\Omega)$ such that
\begin{align*}
\int_\Omega (\vec{p}-A \nabla u)\cdot \vec{v}_1+\vec{p}\cdot \nabla v_2&+B u\,v_2\,dx+\int_{\Gamma_D} u v_3\,ds\\
&= g(v_2)+\int_{\Gamma_N} h_N v_2 \,ds+\int_{\Gamma_D} h_D v_3\,ds
 \end{align*}
for all $(\vec{v}_1,v_2,v_3) \in L_2(\Omega)^d \times H^1_{0,\Gamma_D}(\Omega) \times \widetilde{H}^{-\frac12}(\Gamma_D)$.
\item({\bf 1st order ultra-weak formulation})  \label{fourth}
Assuming $Bw=\vec{b}\cdot \nabla w+cw$ for some $\vec{b} \in L_\infty(\Omega)^d$ and $c \in L_\infty(\Omega)$,
find $(\vec{p},u) \in L_2(\Omega)^d \times L_2(\Omega)$ such that
\begin{align*}
\int_\Omega A^{-1} \vec{p}\cdot \vec{v}_1+u \divv \vec{v}_1 +(\vec{b}\cdot A^{-1} \vec{p}&+c u) v_2+\vec{p}\cdot\nabla v_2\,dx\\
&=
\int_{\Gamma_D} h_D \vec{v}_1 \cdot \vec{n} \, ds+g(v_2)+\int_{\Gamma_N} h_N v_2\,ds
\end{align*}
for all $(\vec{v}_1,v_2) \in H_{0,\Gamma_N}(\divv;\Omega) \times H^1_{0,\Gamma_D}(\Omega)$.
\end{mylist}

\subsection{Finite element discretizations and verification of the uniform inf-sup conditions $\inf_{\delta} \gamma^\delta>0$ and $\inf_{\delta} \mu^\delta>0$} \label{sec:fem}
We assume that $\Omega \subset \R^d$ is a polytope, and let $\{\tria^\delta\}_\delta$ be a family of conforming, uniformly shape regular partitions of $\Omega$ into (closed) $d$-simplices.
With $\cF(\tria^\delta)$ we denote the set of (closed) facets of $K \in  \tria^\delta$.
We assume that $\Gamma_D$, and thus $\Gamma_N$,  is the union of some $e \in \cF(\tria^\delta)$.
We set $\cF_D^\delta:=\{e \in \cF(\tria^\delta)\colon e \subset \Gamma_D\}$, with a similar definition of $\cF_N^\delta$.

For $p \in \N_0$ and $k \in \N_0 \cup \{-1\}$, with $\cS_p^k(\tria^\delta)$ we denote the space of all $C^k$-piecewise polynomials of degree $p$ w.r.t.~$\tria^\delta$.
Spaces  $\cS_p^k(\cF_D^\delta)$ and $\cS_p^k(\cF_N^\delta)$ are defined analogously.

We take $A=\identity$, although the arguments given below apply equally when $A$ is piecewise constant w.r.t.~$\tria^\delta$.
For convenience, we take $B=0$, but the case of $B$ being a PDO of first order with piecewise constant coefficients w.r.t.~$\tria^\delta$ poses no additional difficulties.

For the examples \ref{first}--\ref{fourth} from Sect.~\ref{sec:wellposed} and $d \geq 2$, below 
we discuss the choice of spaces $X^\delta$, $Y^\delta$ and $\hat{X}^\delta$ for the uniform inf-sup conditions to be satisfied.

\begin{mylist}
\item({\bf 2nd order weak formulation}) 
As shown in \cite[\S4.1]{204.19}, for $p \geq 1$ and $X^\delta:=\cS_p^0(\tria^\delta)$,  the choice $Y^\delta:=\big(\cS^0_{p+d-1}(\tria^\delta) \cap H^1_{0,\Gamma_D}(\Omega)\big) \times \cS_p^{-1}(\cF_{\new{D}}^\delta)$ gives $\inf_{\delta} \gamma^\delta>0$. 
\new{Analogous arguments that were used to prove $\inf_{\delta} \gamma^\delta>0$ in \cite[Prop.~4.1]{204.19} show
that $\hat{X}^\delta:=\cS_{p+2d-2}^0(\tria^\delta)$
ensures $\inf_{\delta} \mu^\delta>0$.}
To guarantee that data-oscillation is of higher order, in the definition of $Y^\delta$ one has to replace $\cS^0_{p+d-1}(\tria^\delta)$ by $\cS^0_{p+d}(\tria^\delta)$, and consequently to take $\hat{X}^\delta:=\cS_{p+2d-1}^0(\tria^\delta)$. 
\item({\bf 1st order mild formulation}) As follows from \cite[\S4.2]{204.19}, for $p \geq 1$ and $X^\delta:=\RT_{p-1}(\tria^\delta) \times \cS_p^0(\tria^\delta)$, the choice $Y^\delta:=\cS_p^{-1}(\tria^\delta)^d \times \cS_{p-1}^{-1}(\tria^\delta) \times \new{\cS_p^{-1}(\cF_D^\delta)} \times \new{\big(\cS_{d+p-1}^{0}(\cF_N^\delta)\cap H^1_0(\Gamma_N)\big)}$ gives 
$\inf_{\delta} \gamma^\delta>0$, where data-oscillation is of higher order.
\new{Based on our empirical investigations in \cite{204.19}} of the inf-sup constant $\gamma^\delta$ with the 1st order ultra-weak formulation, we conjecture that, for $d=2$, with $\hat{X}^\delta:=\RT_{p+d-1}(\tria^\delta) \times \cS_{p+d-1}^0(\tria^\delta)$ it holds that $\inf_{\delta} \mu^\delta>0$.
\item({\bf 1st order mild-weak formulation}) As can be deduced from \cite[\S4.3]{204.19}, for $p \geq 1$ and $X^\delta:=\cS^{-1}_{p-1}(\tria^\delta)^d \times \cS_p^0(\tria^\delta)$,
the choice $Y^\delta:=\cS_{p-1}^{-1}(\tria^\delta)^d \times \big(\cS^0_{p+d-1}(\tria^\delta) \cap H^1_{0,\Gamma_D}(\Omega)\big) \times \cS_p^{-1}(\cF^\delta_{\new{D}})$
gives $\inf_{\delta} \gamma^\delta>0$, and a data-oscillation that is of higher order.
Now by taking $\hat{X}^\delta:=\cS^{-1}_{p+d-2}(\tria^\delta)^d \times \cS_{p+d}^0(\tria^\delta)$ one guarantees $\inf_{\delta} \mu^\delta>0$.
\item({\bf 1st order ultra-weak formulation}) As shown in \cite[\S4.4]{204.19}, for $p=0$ the choice $X^\delta:=\cS^{-1}_p(\tria^\delta)^d \times \cS^{-1}_p(\tria^{\new{\delta}})$ and $Y^\delta:=\big(\RT_p(\tria^{\delta}) \times \cS^0_{d+p}(\tria^\delta)\big) \cap Y$, where $Y=H_{0,\Gamma_N}(\divv;\Omega) \times H^1_{0,\Gamma_D}(\Omega)$ gives 
$\inf_{\delta} \gamma^\delta>0$. Numerical experiments indicate that the same holds true for  $p \in \{1,2,3,4\}$ and $d=2$.
This choice of $Y^\delta$ does not guarantee that data oscillation is of higher order, and it appeared that the a posteriori error estimator is not reliable. This problem was solved by taking 
$Y^\delta:=\big(\RT_{p+1}(\tria^{\delta}) \times \cS^0_{d+p+1}(\tria^\delta)\big) \cap Y$.
Since for this formulation $X=L_2(\Omega)^d \times L_2(\Omega) \simeq X'$, there is no need to introduce the additional variable $\theta^\delta$, and so to select a space $\hat{X}^\delta$.
The pair $(\lambda^\delta,u^\delta) \in X^\delta \times Y^\delta$ can be efficiently solved from \eqref{eq:minresprac5}.
Also for this example a remaining difference with \cite{204.19} is that we use the optimal test norm on $Y$ instead of the canonical norm.
\end{mylist}

\section{Numerical experiments}  \label{sec:5}
On a rectangular domain $\Omega = (-1,1)\times (0,1)$ with  Neumann and Dirichlet boundaries $\Gamma_N=[-1,0]\times\{0\}$ and $\Gamma_D=\overline{\partial \Omega \setminus \Gamma_N}$, for $g \in H^1_{0,\Gamma_D}(\Omega)'$, $h_D \in H^{\frac12}(\Gamma_D)$, and $h_N \in H^{-\frac12}(\Gamma_N)$,
we consider the Poisson problem of finding $u \in H^1(\Omega)$ that satisfies
$$
 \left\{
\begin{array}{r@{}c@{}ll}
-\Delta u&\,\,=\,\,& g &\text{ on } \Omega,\\
u &\,\,=\,\,& h_D &\text{ on } \Gamma_D,\\
\nabla u \cdot \vec{n}&\,\,=\,\,& h_N &\text{ on } \Gamma_N.
\end{array}
\right.
$$
We prescribe the solution $u(r,\theta) := r^{\frac12} \sin \frac{\theta}{2}$ in polar coordinates, and determine the data correspondingly. Then $g=0$, $h_N=0$, and $h_D=0$ on $[0,1]\times\{0\}$, but $h_D \neq 0$ on the remaining part of $\Gamma_D$.
It is known that $u \in H^{\frac32-\eps}(\Omega)$ for all $\eps>0$, but $u \not\in H^{\frac32}(\Omega)$ (\cite{77}).

We consider above problem in the first order mild-weak formulation. Compared to the second order formulation, this first order formulation does not require additional smoothness conditions on the data. Furthermore, the norms for solution and flux variables are balanced, in the sense that given some regularity of the solution, both variables can qualitatively be equally well approximated by finite element functions.
For this formulation the Neumann boundary condition is a natural one, but the Dirichlet  boundary condition is essential, and is therefore imposed by an additional variational equation.

We consider a family of conforming triangulations $\{\mathcal{T}^\delta\}_{\delta}$ of $\Omega$, where each triangulation is created using newest vertex bisections starting from an initial triangulation that consists of 8 triangles created by first cutting $\Omega$ along the y-axis and then cutting the resulting two squares along their diagonals. The interior vertex of the initial triangulation of both squares are labelled as the `newest vertex' of all four triangles in both squares.
Following Sect.~\ref{sec:fem}, given some polynomial degree $p \geq 1$, we set
\begin{align*}
X^\delta&:=\cS^{-1}_{p-1}(\tria^\delta)^2 \times \cS_p^0(\tria^\delta),\\
Y^\delta&:=\cS_{p-1}^{-1}(\tria^\delta)^2 \times \big(\cS^0_{p+1}(\tria^\delta) \cap H^1_{0,\Gamma_D}(\Omega)\big) \times \cS^{-1}_{p}(\cF^\delta_{\new{D}}),\\
\hat{X}^\delta&:=\cS^{-1}_{p}(\tria^\delta)^2 \times \cS_{p+2}^0(\tria^\delta).
\end{align*}

With $(G(\vec{p}, u))(\vec{v}_1,v_2,v_3) := \int_\Omega (\vec{p}- \nabla u)\cdot \vec{v}_1+\vec{p}\cdot \nabla v_2\,dx+\int_{\Gamma_D} u v_3\,ds$, our practical MINRES method computes $(\vec{\theta}_1^\delta, \theta_2^\delta,  \vec{\lambda}_1^\delta,\lambda_2^\delta,\lambda_3^\delta,\vec{p}^\delta, u^\delta) \in \hat{X}^\delta\times Y^\delta\times X^\delta$
such that 
\begin{align*}
\langle (\vec{\theta}_1^\delta, \theta_2^\delta), (\undertilde{\vec{\theta}}_1^\delta, \undertilde{\theta}_2^\delta) \rangle_{L_2(\Omega)^2\times H^1(\Omega)}& - (G(\undertilde{\vec{\theta}}_1^\delta, \undertilde{\theta}_2^\delta)) ( \vec{\lambda}_1^\delta,\lambda_2^\delta,\lambda_3^\delta)
- (G(\vec{\theta}_1^\delta, \theta_2^\delta)(\undertilde{\vec{\lambda}}_1^\delta,\undertilde{\lambda}_2^\delta,\undertilde{\lambda}_3^\delta) \\
& - (G(\vec{p}^\delta, u^\delta))(\undertilde{\vec{\lambda}}_1^\delta,\undertilde{\lambda}_2^\delta,\undertilde{\lambda}_3^\delta) - (G(\undertilde{\vec{p}}^\delta, \undertilde{u}^\delta))(\vec{\lambda}_1^\delta,\lambda_2^\delta,\lambda_3^\delta) \\
	= -g(\undertilde{\lambda}_2^\delta)&-\int_{\Gamma_N} h_N \undertilde{\lambda}_2^\delta \,ds-\int_{\Gamma_D} h_D \undertilde{\lambda}_3^\delta\,ds =: -f(\undertilde{\lambda}_2^\delta,\undertilde{\lambda}_3^\delta)
\end{align*}
for all $(\undertilde{\vec{\theta}}_1^\delta, \undertilde{\theta}_2^\delta, \undertilde{\vec{\lambda}}_1^\delta,\undertilde{\lambda}_2^\delta,\undertilde{\lambda}_3^\delta,
\undertilde{\vec{p}}^\delta, \undertilde{u}^\delta) \in \hat{X}^\delta\times Y^\delta\times X^\delta$.

The method comes with a built-in a posteriori error estimator $\cE(\vec{p}^\delta,u^\delta,f):=\sqrt{\sum_{T \in \tria^\delta} \|\vec{\theta}_1^\delta\|_{L_2(T)^2}^2+\|\theta_2^\delta\|_{H^1(T)}^2}$ (see Sect.~\ref{sec:apost}), which is efficient and, because the data-oscillation term is of higher order than the best approximation error, in any case asymptotically reliable.

For $p \in \{1,2,3\}$ we performed numerical experiments with uniform and adaptively refined triangulations. Concerning the latter, we have used the element-wise error indicators $\sqrt{\|\vec{\theta}_1^\delta\|_{L_2(T)^2}^2+\|\theta_2^\delta\|_{H^1(T)}^2}$ to drive an AFEM with D\"{o}rfler marking with marking parameter $\theta=0.6$.
The results given in Figure~\ref{fig1} show that for uniform refinements increasing $p$ does not improve the order of convergence, because it is limited by the low regularity of the solution in the hilbertian Sobolev scale. However, we observe that the adaptive routine attains the best possible rates allowed by the order of approximation of $X^\delta$.

\begin{figure}[h!]
\hspace*{0cm}
\begin{subfigure}{0.5\textwidth}
\centering
\includegraphics[width=\linewidth]{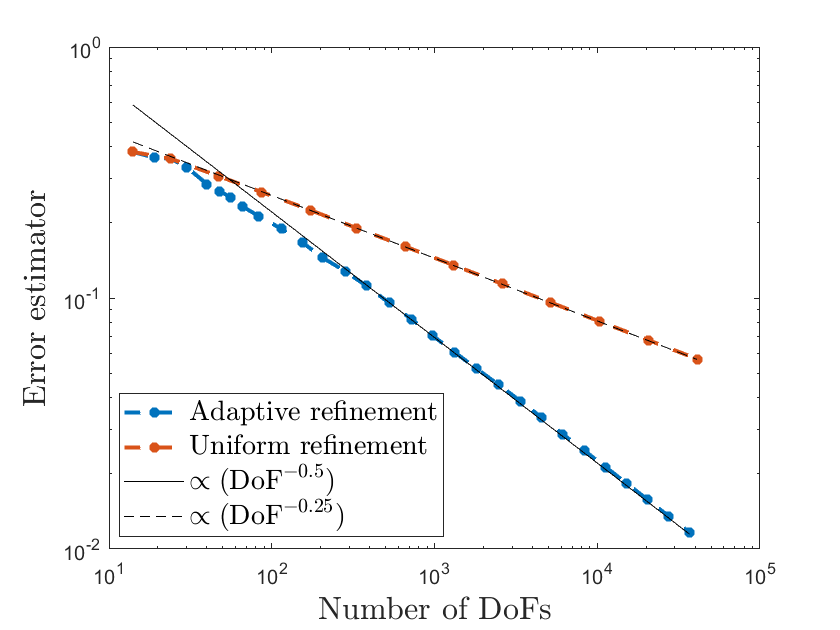}

\end{subfigure}\hspace*{0cm}%
\begin{subfigure}{0.5\textwidth}
\centering
\includegraphics[width=\linewidth]{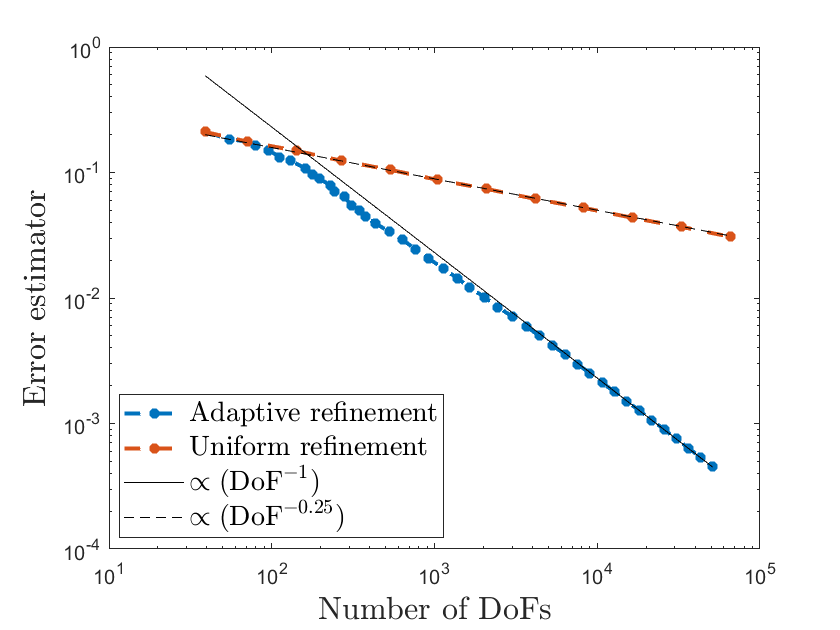}
\end{subfigure}

\begin{subfigure}{0.5\textwidth}
\centering
\includegraphics[width=\linewidth]{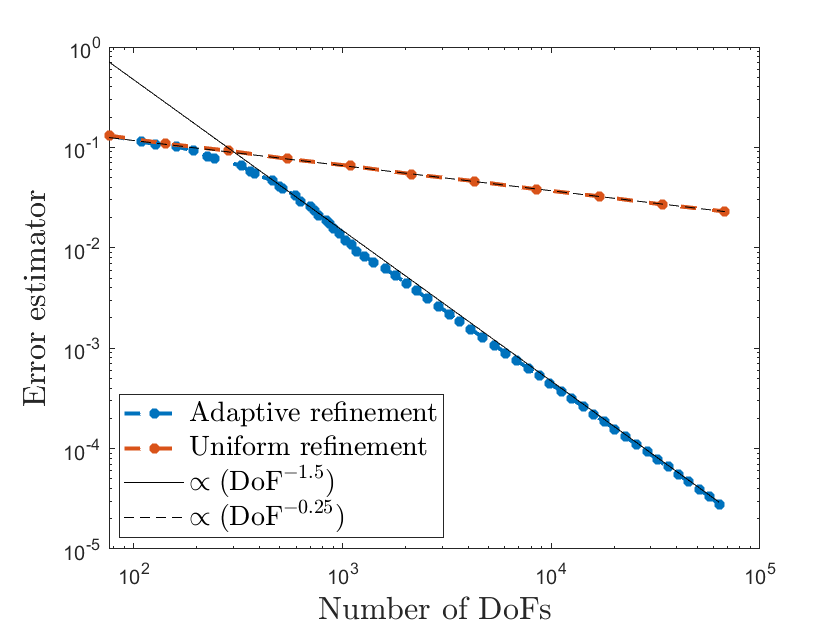}

\end{subfigure}\hspace*{0cm}%
\begin{subfigure}{0.5\textwidth}
\centering
\includegraphics[width=\linewidth]{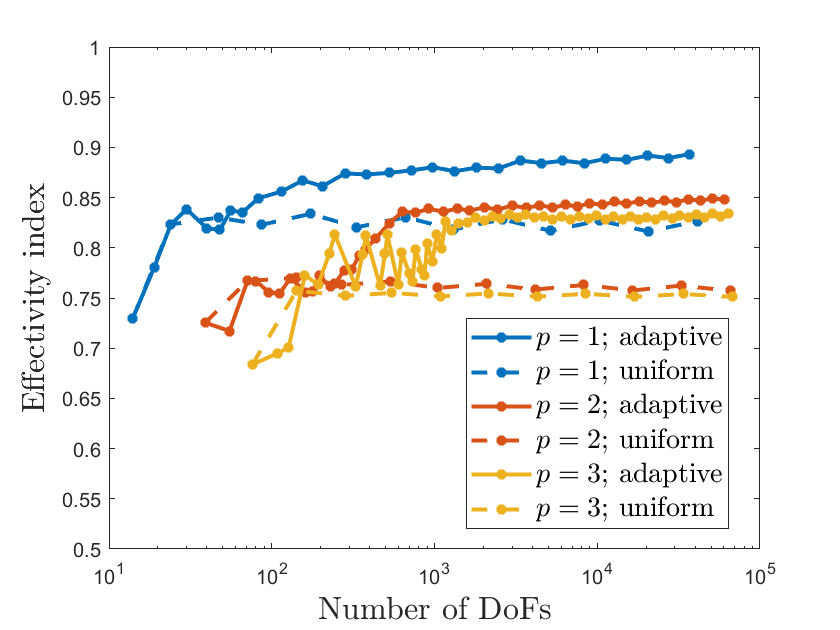}
\end{subfigure}
\caption{Number of DoFs in $X^\delta$ vs.~$\cE(\vec{p}^\delta,u^\delta,f)$:
 left-upper: $p=1$, right-upper: $p=2$, left-bottom: $p=3$. \\
 Right-bottom: number of DoFs in $X^\delta$ vs. effectivity index $\cE(\vec{p}^\delta,u^\delta,f)/\sqrt{\|\vec{p}-\vec{p}^\delta\|_{L_2(\Omega)^2}^2+\|u-u^\delta\|_{H^1(\Omega)}^2}$ .}
\label{fig1}
\end{figure}

\subsection*{Acknowledgement} The author is indebted to Harald Monsuur for the computation of the numerical results.

\newcommand{\etalchar}[1]{$^{#1}$}

\end{document}